\documentclass{article}

\usepackage[utf8]{inputenc}
\usepackage{amsmath}
\usepackage{amsfonts}
\usepackage{amsthm}
\usepackage{myarticles}
\usepackage{bbm}
\usepackage{cases}
\usepackage{enumitem}
\usepackage{mathtools}
\usepackage{leftidx}
\usepackage[
backend=biber,
style=numeric,
sorting=nyt,
abbreviate=true,
natbib=true,
sortlocale=en_US,
isbn=false,
first inits=true,
url=false, 
doi=true,
eprint=false
]{biblatex}
\usepackage[]{todonotes}

\theoremstyle{Mio}
\newtheorem{prop}{Proposition}[section]
\newtheorem{lemma}{Lemma}[section]
\newtheorem{teo}{Theorem}[section]

\newtheorem{innercustomthm}{Assumption}
\newenvironment{Asmpt}[1]
{\renewcommand\theinnercustomthm{#1}\innercustomthm}
{\endinnercustomthm}

\theoremstyle{remark}
\newtheorem{rem}{Remark}[section]

\newtheorem{defi}{Definition}[section]
\date{}

\title{Optimal stopping of marked point processes and
 reflected backward stochastic differential equations}
\author{Nahuel Foresta\thanks{
Dipartimento di Matematica, Politecnico di Milano, via Bonardi 9, 20133 Milano, Italy. nahueltomas.foresta@polimi.it. The
 author was supported
by the Italian MIUR-PRIN 2015 ``Deterministic and stochastic evolution equations'' and
INDAM-GNAMPA.}
}
\begin{document}
    \maketitle
    \begin{abstract}
        We define a class of reflected backward stochastic differential equation (RBSDE)
        driven by a marked point process (MPP) and a Brownian motion,
        where the solution is constrained to stay above a given c\`adl\`ag
        process.
The MPP is only required to be non-explosive and to have
totally inaccessible jumps.
        Under suitable assumptions
        on the coefficients we obtain existence and uniqueness of the solution, using the Snell envelope theory. We use the equation to represent the value function of an optimal stopping problem, and we characterize the optimal strategy.\\[1cm]
    \end{abstract}
    \section{Introduction}
    Nonlinear backward stochastic differential equations (BSDE) driven by a Brownian motion were first introduced by Pardoux and Peng in the seminal paper \cite{pardoux1990adapted}.
    Later, BSDE have found applications in several fields of mathematics, such as
    stochastic control, mathematical finance, nonlinear PDEs
    (see for instance \cite{el1997backward,Pardoux2014book,Crepey2013book}).
As the driving noise, the Brownian motion has been replaced
by more general classes of martingales; the first example
is perhaps \cite{elkaroui97generalbsde}, see \cite{jianming2000backward}
for a very general situation.

In particular, occurrence of marked point processes in the equation
has been considered since long.
    In \cite{tang1994necessary,barles1997backward}, related to optimal control
    and PDEs respectively,
    an independent Poisson random measure is added to the driving Wiener noise.
Motivated by several applications to stochastic optimal control
and financial modelling, more general marked point processes
were considered in the BSDE.
Examples can be found in \cite{Becherer2006,confortola2013backward}
for $L^2$ solutions, \cite{confortola2014backward} for the $L^1$ case
and   \cite{confortola16LP}
for the $L^p$ case.

    In connection with optimal stopping and obstacle problems,
    in \cite{el1997reflected} a reflected BSDE is introduced,
    where the solution is forced to stay above a certain continuous
    barrier process. This class of BSDE finds applications in various
    problems in finance and stochastic games theory.
A number of generalizations has followed, both with variations on the nature
of the barrier process and the type of noise. In the Brownian case,
in \cite{hamadene2002reflected} the author solves the problem
when the obstacle
is just càdlàg
in \cite{peng2005smallest}, the authors allow the obstacle
to be only $L^2$.
On the other hand,
in \cite{hamadene2003reflected} the authors solve the problem when a Poisson
noise is added, and the barrier is càdlàg with inaccessible jump times.
This is later generalized in \cite{hamadene2016reflected} where the barrier
can have partially accessible jumps too. Other specific results are \cite{essaky11}
where a BSDE with two generators is solved in a Wiener framework
and \cite{ren10} in a Lévy framework; the papers \cite{ren08} and \cite{elotmani09}
where the noise is a Teugels Martingale associated to a
one-dimensional
Lévy process.
The paper \cite{crepey08refelctedcomparison} that considers
a marked point process with compensator
admitting a bounded desnity with respect to the Lebesgue measure.

Finally, very general barriers beyond the càdlàg case
were recently considered in   \cite{grigorova15,grigorova2016optimal}.

It is the aim of
the present work to address the case when
the obstacle to be a càdlàg process and, in addition
to the Wiener process, a very general
marked point process occurs in the equation. The only assumptions
we make is that it is non-explosive and has totally inaccessible
jumps. This is equivalent to the requirement
that the compensator of the counting process of the jumps
has continuous trajectories. However, we do not require
absolute continuiuty with respect to the Lebegue measure.
To our knowledge, only in \cite{Bandini2015nonquasi,bandini2017optimal}, in \cite{papapantoleon2016existence}, and in 	\cite{cohen10generalcomparison,cohen12generalspaces}
 even more general cases have been addressed, but
 without reflection.
    
    The equation has the form
    \begin{align}\begin{split}
    Y_t=\xi+\int_t^Tf_s(Y_s,U_s)dA_s&+\int_t^Tg_s(Y_s,Z_s)ds\\&-\int_t^T\int_EU_s(e)q(dtde)-\int_t^TZ_sdW_s+K_T-K_t\\
    &Y_t\geq h_t.
    \end{split}\end{align}
    Here $W$ is a Brownian motion and $q$, independent from $W$, is a compensated
    integer random measure corresponding to some
    marked point process $(T_n,\xi_n)_{n\geq 1}$:
    see \cite{bremaud1981point,jacod1975multivariate,lastbrandt1995}
    as general references on the subject.
     The data are the final condition $\xi$ and the generators $f$ and $g$. $A$ is a continuous stochastic increasing process related to the point process. The $Y$ part of the solution is constrained to stay above a given barrier process $h$, and the $K$ term is there to assure this condition holds
    This equation is then used to solve a non-markovian optimal stopping problem, where the running gain, stopping reward and final reward are the data used in the BSDE. Under additional assumptions on the barrier process, an optimal stopping time is characterized.

    This work generalizes the results previously obtained by allowing
    a more general structure in the jump component. This introduces
    some technical difficulties and some assumptions. For instance,
    we work in ``weighted $L^2$ spaces", with a weight of the form
    $e^{\beta A_t}$, and the data must satisfy this integrability conditions.
    Direct use of standard tools, like the Gronwall lemma,
    becomes difficult
    in our case, so we have to resort to direct estimates. Since there is no
    general comparison theorem for BSDE with so general marked point process, we do
    not use a penalization method, but rather a combination of the Snell envelope
    theory and contraction theorem.

    The paper is organized as follows: in section \ref{sec:framework_obj}
    we first recall some results on marked point processes and describe the
    setting and the problem we want to solve. In section \ref{sec:known_gens}
    we prove the existence and uniqueness of a Reflected BSDE driven by a
    marked point process and a Wiener process when the generators do
    not depend on the solution of the BSDE.  This is solved in some $L^2$ space,
    appropriate for the Brownian motion. When the (given) generator and the
    other data are adapted only to the filtration generated by the point process, the solution can be found in a larger space. We then link  these equations to an optimal stopping problem.
    Lastly in section \ref{sec:general_BSDE} we solve the BSDE in the general case
 with the help of a contraction argument. Here we use the $L^2$ framework
 for both the case with only marked point process or with both driving processes.

    \section{Preliminaries, assumptions, formulation of the problems}
    \label{sec:framework_obj}
    \subsection{Some reminders on point processes}
    We start by recalling some notions about marked point processes and then defining
    the objectives of this paper. For a comprehensive treatment of marked point processes,
    we refer the reader to \cite{jacod1975multivariate}, \cite{bremaud1981point}
    or \cite{lastbrandt1995}.
    Let $(\Omega,\mathcal{F},\mathbb{P})$ be a complete probability space and
    let $E$ be a Borel space, i.e. a topological space homeomorphic to a Borel subset
    of a compact metric space (sometimes called Lusin space; we recall that every
separable complete metric space is Borel). We call $E$ the mark space and we denote by
    $\mathcal{E}$ its Borel $\sigma$-algebra.

    \begin{defi}
        A marked point process (MPP) is a sequence of random variables $(T_n,\xi_n)_{n\geq 0}$ with values in $[0,+\infty]\times E$ such that $\mathbb{P}$-a.s.
        \begin{itemize}
            \item $T_0=0$.
            \item $T_n\leq T_{n+1} \forall n\geq 0$.
            \item $T_n<\infty\Rightarrow T_n<T_{n+1} \forall n\geq 0$.
        \end{itemize}

    \end{defi}
    We will always assume the marked point process in the paper to be non-explosive, that is $T_n\rightarrow+\infty$ $\mathbb{P}$-a.s.
    To each marked point process we associate a random discrete measure $p$ on $((0,+\infty)\times E,\mathcal{B}((0,+\infty)\otimes\mathcal{E})$:
$$
p(\omega,D)=\sum_{n\geq 1}\ind_{(T_n(\omega),\xi_n(\omega))\in D}.
$$

We refer to $p$ also as marked point process. For each $C\in\mathcal{E}$, define the
counting process $N_t(C)=p((0,t]\times C)$ that counts how many jumps have occurred to $C$ up to time $t$.
Denote $N_t=N_t(E)$. They are right continuous increasing process starting from zero.
Each point process generates a filtration $\mathbb{G}=(\mathcal{G}_t)_{t\geq 0}$ as follows: define for $t\geq 0$
$$
\mathcal{G}_t^0=\sigma(N_s(C)\; :\; s\in[0,t], C\in\mathcal{E})
$$
and set $\mathcal{G}_t=\sigma(\mathcal{G}_t^0,\mathcal{N})$, where $\mathcal{N}$
is the family of $\mathbb{P}$-null sets of $\mathcal{F}$. $\mathbb{G}$ is a
right-continuous filtration that satisfies the usual hypotheses. Denote by $\mathcal{P}^{\mathcal{G}}$
the $\sigma$-algebra of $\mathcal{G}$-predictable processes.

For each marked point process there exists a unique predictable random measure $\nu$,
called compensator, such that for all non-negative $\mathcal{P}^{\mathcal{G}}\otimes\mathcal{E}$-measurable
 process $C$ it holds that
 $$
 \evals{\int_0^{+\infty}\int_EC_t(e)p(dtde)}=\evals{\int_0^{+\infty}\int_EC_t(e)\nu(dtde)}.
 $$
Similarly,  there exists a unique right continuous increasing process with $A_0=0$, the dual predictable projection of $N$, such that for all non-negative predictable processes $D$
$$
\evals{\int_0^{+\infty}D_tdN_t}=\evals{\int_0^{+\infty}D_tdA_t}.
$$
It is known that there exists a function $\phi$ on $\Omega\times[0,+\infty)\times\mathcal{E}$ such that we have the disintegration  $\nu(\omega,dtde)=\phi_t(\omega,de)dA_t(\omega)$. Moreover the following properties hold:
\begin{itemize}
    \item for every $\omega\in\Omega$, $t\in[0,+\infty)$,  $C\mapsto\phi_t(\omega,C)$ is a probability on $(E,\mathcal{E})$.
    \item for every $C\in\mathcal{E}$, the process $\phi_t(C)$ is predictable.
\end{itemize}
We will assume in the following that all marked point processes in this paper have a compensator of this form.\\
From now on, fix a terminal time $T>0$. Next we need to define integrals with respect to point processes.
\begin{defi}
    Let $C$ be a $\mathcal{P}^{\mathcal{G}}\otimes\mathcal{E}$-measurable process such that
    $$
    \evals{\int_0^T\int_E |C_t(e)|\phi_t(de)dA_t}<\infty.
    $$
    Then we can define the integral
    $$
    \int_0^T\int_E C_t(e)q(dtde)=\int_0^T\int_EC_t(e)p(dtde)-\int_0^T\int_EC_t(e)\phi_t(de)dA_t
    $$
    as difference of ordinary integrals with respect to $p$ and $\phi dA$.
\end{defi}
\begin{rem}
    In the paper we adopt the convention that $\int_a^b$ denotes an integral on $(a,b]$ if $b<\infty$, or on $(a,b)$ if $b=\infty$.
\end{rem}
\begin{rem}
    Since $p$ is a discrete random measure, the integral with respect to $p$ is a sum:
    $$
    \int_0^t\int_E C_s(e)p(dsde)=\sum_{T_n\leq t}C_{T_n}(\xi_n)
    $$
\end{rem}
Given a process $C$ as above, the integral defines a process $\int_0^t\int_E C_s(e)q(dsde)$ that, by the definition of compensator, is a martingale.\\

\subsection{Probabilistic setting}
In this paper we will assume that $(\Omega,\mathcal{F},\prob)$ is a complete probability space
and $p(dtdx)$   a marked point process on a Borel space $(E,\mathcal{E})$ as before,
 whose compensator is $\phi_t(dx)dA_t$. In addition we assume we are given
   an independent Wiener process $W$ in $\mathbb{R}^d$.
    Let $\mathbb{G}=\left(\mathcal{G}_t\right)_{t\geq 0}$
    (resp. $\mathbb{F}=\left(\mathcal{F}_t\right)_{t\geq 0}$) be the
    completed filtration generated by $p$ (resp. $p$ and $W$), which satisfies the usual conditions.
    Let $\mathcal{T}_t$ be the set of $\mathbb{F}$-stopping times greater than $t$.
    Denote by $\mathcal{P}$ ({resp.} $Prog$) be the predictable ({resp.} progressive)
     $\sigma$-algebra relative to $\mathbb{F}$. For $\beta>0$, we introduce the following spaces of equivalence classes we will be using in the following

\begin{itemize}
    \item $L^{r,\beta}(A)$ ({resp.} $L^{r,\beta}(A,\mathbb{G})$) is the space of all
    $\mathbb{F}$-progressive ({resp.} $\mathbb{G}$-progressive) processes $X$ such that
    $$||X||_{L^{r,\beta}(A)}^r=\eval\left[\int_0^Te^{\beta A_s}|X_s|^rdA_s\right]<\infty.$$
    \item $L^{r,\beta}(p)$ ({resp.} $L^{r,\beta}(p,\mathbb{G})$)
    is the space of all $\mathbb{F}$-predictable ({resp.} $\mathbb{G}$-predictable)
    processes $U$ such that $$||U||_{L^{r,\beta}(p)}^r
    =\eval\left[\int_0^T\int_E e^{\beta A_s}|U_s(e)|^r\phi_s(de)dA_s\right]<\infty.$$
    \item $L^{r,\beta}(W,\mathbb{R}^d)$ ({resp.} $L^{r,\beta}(W,\mathbb{R}^d,\mathbb{G})$)
    is the space of $\mathbb{F}$-progressive ({resp.} $\mathbb{G}$-progressive)
    processes $Z$ in $\mathbb{R}^d$ such that $$||Z||_{L^{r,\beta}(W)}^r=\eval\left[\int_0^Te^{\beta A_s}|Z_s|^rds\right]<\infty$$
    \item $\mathcal{I}^2$ ({resp.} $\mathcal{I}^2(\mathbb{G})$)
    is the space of all càdlàg increasing $\mathbb{F}$-predictable ({resp.}
    $\mathbb{G}$-predictable) processes $K$ such that $\eval[K^2_T]<\infty$.
\end{itemize}
One last tool we will need in the following is the martingale representation theorem:
if $M$ is a  càdlàg  square integrable $\mathbb{F}$-martingale on $[0,T]$, then there exist two processes $U$ and $Z$ such that
\begin{gather*}
\evals{\int_0^T\int_E|U_t(e)|\phi_t(de)dA_t}+\evals{\int_0^T|Z_t|^2dt}<\infty\\
M_t=M_0+\int_0^t\int_EU_s(e)\phi_s(de)dA_s+\int_0^tZ_sdW_s.
\end{gather*}

\subsection{Assumptions and formulation of the problem}

Let $(\Omega,\mathcal{F},\prob)$, $(E,\mathcal{E})$,  $p(dtdx)$, $W$, $\mathbb{F}$ be as before.
We will consider the following reflected BSDE.

        \begin{equation}
    \label{eq:system_Y_both}
    \begin{cases}
    Y_t=\xi+\int_t^Tf(s,Y_s,U_s)dA_s+\int_t^Tg(s,Y_s,Z_s)ds-\int_t^T\int_EU_s(y)q(dsdy)\\
    \qquad-\int_t^TZ_s(y)dW_s+K_T-K_t, \qquad \forall t\in[0,T] \text{ a.s.}\\
    Y_t\geq h_t, \qquad \forall t\in[0,T] \text{ a.s.}\\
    \int_0^T(Y_{s}-h_{s})dK^c_s=0 \text{ and } \Delta K_t\leq(h_{t^-}-Y_t)^+\ind_{\lbrace Y_{t^-}=h_{t^-}\rbrace}\forall t\in[0,T] \text{ a.s.},
    \end{cases}
    \end{equation}
    A solution   is a quadruple $(Y,U,Z,K)$ that lies in
    $ \left(L^{2,\beta}(A)\cap L^{2,\beta}(W)\right)\times L^{2,\beta}(p)\times L^{2\beta}(W)\times \mathcal{I}^2$, with $Y$ càdlàg,
    that satisfies \eqref{eq:system_Y_both}. The condition on the last line in \eqref{eq:system_Y_both} is called
the Skorohod condition, or
 the minimal push condition. It can be expressed in an alternative way: see Remark
\ref{skorohodalt} below.

    Let us now state the general assumptions that will be used throughout the paper.
Additional specific assumptions will be presented  in section \ref{sec:general_BSDE}.
    The first one is an assumption on the compensator $A$ of the counting process $N$ relative to $p$.
    \begin{Asmpt}{(A)} \label{Ass:compensatore}
        The process $A$ is continuous. \end{Asmpt}

    \begin{Asmpt}{(B)}\  \\
        \label{Ass:data}
        \begin{enumerate}[label=\roman{*})]
            \item
            The final condition $\xi:\Omega\rightarrow\mathbb{R}$ is $\mathcal{F}_T$-measurable and
            $$\eval\left[e^{\beta A_T}\xi^2\right]<\infty.$$

            \item For every $\omega\in\Omega$, $t\in\left[0,T\right]$, $r\in\mathbb{R}$ a
            mapping $$f(\omega,t,r,\cdot):{L}^2(E,\mathcal{E},\phi_t(\omega,dy))\rightarrow \mathbb{R}$$ is given and satisfies the following:
            \begin{enumerate}[label=\alph{*})]
                \item for every $U\in{L}^{2,\beta}(p)$ the mapping
                $$
                (\omega,t,r)\mapsto f(\omega,t,r,U_t(\omega,\cdot))
                $$
                is $\mathit{Prog}\otimes\mathcal{B}(\mathbb{R})$-measurable, where $\mathit{Prog}$
                denotes the progressive $\sigma$-algebra.
                \item There exist $L_f\geq 0$, $L_U\geq 0$ such that for every
                $\omega\in\Omega$, $t\in\left[0,T\right]$, $y,y'\in\mathbb{R}$,
                $u,u'\in{L}^{2}(E,\mathcal{E},\phi_t(\omega,dy))$ we have
                \begin{multline*}
                |f(\omega,t,y,u(\cdot))-f(\omega,t,y',u'(\cdot))|\leq\\ L_f|y-y'|
                +L_U\left(\int_E|u(e)-u'(e)|^2\phi_t(\omega,de)\right)^{1/2}
                \end{multline*}
                \item we have
                $$
                \eval\left[\int_0^Te^{\beta A_s}|f(s,0,0)|^2dA_s\right]<\infty.
                $$
            \end{enumerate}
            \item The mapping $g:\Omega\times[0,T]\times \mathbb{R}\times\mathbb{R}^d\rightarrow\mathbb{R}$ is given
            \begin{enumerate}[label=\alph{*})]
                \item $g $ is $Prog\times\mathcal{B}(\mathbb{R})\times\mathcal{B}(\mathbb{R}^d)$ measurable.
                \item There exist $L_g\geq 0$, $L_Z\geq 0$ such that for every $\omega\in\Omega$,
                $t\in\left[0,T\right]$, $y,y'\in\mathbb{R}$, $z,z'\in\mathbb{R}^d$
                $$
                |g(\omega,t,y,z)-g(\omega,t,y',z')|\leq L_g|y-y'|+L_Z|z-z'|
                $$
                \item we have
                $$
                \eval\left[\int_0^Te^{\beta A_s}|g(s,0,0)|^2ds\right]<\infty.
                $$
                \end{enumerate}

            \item $h$ is a càdlàg $\mathbb{F}$-adapted process such that $h_T\leq \xi$. There exists a $\delta>0$
            such that
            $$
            \eval[\sup\limits_{t\in[0,t]}e^{\left(\beta+\delta\right)A_t}h_{t}^2]
            $$
        \end{enumerate}
    \end{Asmpt}

\begin{rem}
We recall that Assumption \ref{Ass:compensatore}  is
equivalent to the fact that the jumps of the point process are totally inaccessible (relative to
$\mathbb{F}$): see
\cite{sheng1998semimartingale} Corollary 5.28.
    We will often use the following consequence:
 since $K$ is required to be predictable, its jumps (that are all non-negative)
 are disjoint from the jumps of $p$;
 so at any jump time of $K$ we also have a jump of $Y$ with the same size,
 but of opposite sign, in symbols we have a.s.
\begin{equation}\label{separatedjumps}
    \Delta K_t \indb{\Delta K_t>0}=(-\Delta Y_t)^+\indb{\Delta K_t>0},
\quad t>0.
\end{equation}
\end{rem}

\begin{rem}\label{skorohodalt}
 The Skorohod condition on the last line in \eqref{eq:system_Y_both}
 tells us that the process $K$ grows only when the solution is about to touch the  barrier.
 We claim that it is in fact equivalent to
        \begin{equation}\label{skorohodequivalent}
    \int_0^T(Y_{s^-}-h_{s^-})dK_s=0, \quad a.s.
\end{equation}
        To check the equivalence, note first that
        $$\int_0^T(Y_{s^-}-h_{s^-})dK_s=\int_0^T(Y_{s}-h_{s})dK^c_s+\sum_{0<s\leq T}^{}(Y_{s^-}-h_{s^-})\Delta K_s
        , \quad a.s.$$
If the Skorohod condition in \eqref{eq:system_Y_both} holds then both terms in the right-hand
side are
          zero, since jumps of $K$ can only happen when $Y_{t^-}=h_{t^-}$.
        Conversely, assume that \eqref{skorohodequivalent}
        holds. Then clearly
        $
        \int_0^T(Y_{s^-}-h_{s^-})dK^c_s=0
        $
and so
        $\int_0^T(Y_{s}-h_{s})dK^c_s=0$. Also,
        $\sum_{0<s\leq T}^{}(Y_{s^-}-h_{s^-})\Delta K_s=0$,
        so $\{t:\Delta K_t>0\}\subset\{t:Y_{t^-}=h_{t^-}\}$
        and, recalling \eqref{separatedjumps},
        we have a.s.
        \begin{align*}
        \Delta K_t=\Delta K_t\indb{\Delta K_t>0}&=(-\Delta Y_t)^+ \indb{\Delta K_t>0}\le
         (-\Delta Y_t)^+ \indb{Y_{t^-}=h_{t^-}}\\&=(Y_{t^-}-Y_t)^+\indb{Y_{t^-}=h_{t^-}}=(h_{t^-}-Y_t)^+\indb{Y_{t^-}=h_{t^-}}.
        \end{align*}
    \end{rem}

\begin{rem} In the simpler case when there is no Brownian component the reflected BSDE
\eqref{eq:system_Y_both} becomes
        \begin{equation}
    \label{eq:system_Y_only_MPP}
    \begin{cases}
    Y_t= {\xi}+\int_t^T {f}(s,Y_s,U_s)dA_s-\int_t^T\int_EU_s(y)q(dsdy)+K_T-K_t,
\: \forall t\in[0,T] \text{ a.s.}
    \\
     Y \text{ càdlàg and }Y\in L^{2,\beta}(A,\mathbb{G}),\quad U\in L^{2,\beta}(p,\mathbb{G}),
     \quad K \in \mathcal{I}^2(\mathbb{G}) \\
    Y_t\geq  {h}_t \qquad \forall t\in[0,T] \text{ a.s.}\\
    \int_0^T(Y_{s}- {h}_{s})dK^c_s=0 \text{ and }
    \Delta K_t\leq( {h}_{t^-}-Y_t)^+\ind_{\lbrace Y_{t^-}= {h}_{t^-}\rbrace} \forall t\in[0,T] \text{ a.s.}
    \end{cases}
    \end{equation}
    Here we only assume we are given the space $(\Omega,\mathcal{F},\prob)$ and the
     marked point process $p$.
The assumptions we need are the same as  in \ref{Ass:compensatore} and \ref{Ass:data}, provided
we set $g=0$ and   $\mathbb{G}=\mathbb{F}$.
\end{rem}

\section{Reflected BSDE with given generators and optimal stopping problem}
\label{sec:known_gens}
In this section we  first study the reflected BSDE in the case when the generators $g$ and $f$ do not depend on
$(Y,Z,U)$ but are a given processes that satisfy
\begin{Asmpt}{(B$^\prime$)}
    \label{Ass:givengen}
    $f$ and $g$ are 
    $\mathbb{F}$-progressive processes such that
    \begin{equation}
    \eval\left[\int_0^Te^{\beta A_s}|f_s|^2dA_s+\int_0^Te^{\beta A_s}|g_s|^2 ds\right]<\infty.
    \end{equation}
\end{Asmpt}
Equation \eqref{eq:system_Y_both} reduces to
        \begin{equation}
\label{eq:simple_Y}
\begin{cases}
Y_t=\xi+\int_t^Tf_sdA_s+\int_t^Tg_sds-\int_t^T\int_EU_s(y)q(dsdy)-\int_t^TZ_sdW_s+K_T-K_t\\
Y\in L^{2,\beta}(A)\cap L^{2,\beta}(W),\quad U\in L^{2,\beta}(p),\quad Z\in L^{2,\beta}(W),\quad K \in \mathcal{I}^2 \\
Y_t\geq h_t \qquad \forall t\in[0,T] \text{ a.s.}\\
\int_0^T(Y_{s}-h_{s})dK^c_s=0 \text{ and }
\Delta K_t\leq(h_{t^-}-Y_t)^+\ind_{\lbrace Y_{t^-}=h_{t^-}\rbrace}\forall t\in[0,T] \text{ a.s.}
\end{cases}
\end{equation}

In this case, the solution $Y$ to the equation is also the value function of an optimal stopping problem, as we will see later.
First we define the càdlàg process $\eta_t$ as

\begin{equation}
\label{eq:etat}
\eta_t=\int_0^{t\wedge T} f_sdA_s+\int_0^{t\wedge T} g_sds+h_t \indb{t<T}+\xi\indb{t\geq T}
\end{equation}
\begin{rem}\label{stimessus}
In the following we will often use this kind of inequalities:
    \begin{multline}
    \hspace{-0.3cm}\left(\int_0^tf_sdA_s\right)^2=\left(\int_0^te^{-\beta A_s/2}e^{\beta A_s/2}|f_s|dA_s\right)^2
    \leq\int_0^te^{-\beta A_s}dA_s\int_0^te^{\beta A_s}f_s^2dA_s\\=\frac{1-e^{\beta A_t}}{\beta}\int_0^te^{\beta A_s}f_s^2dA_s\leq \frac{1}{\beta}\int_0^te^{\beta A_s}f_s^2dA_s
    \end{multline}
\end{rem}
\begin{lemma}\label{etains2}
    Under assumptions \ref{Ass:data}-(i)(iv) and \ref{Ass:givengen}, $\eta$ is of class $[D]$ and
    $$
    \eval\left[\sup\limits_{0\leq t \leq T}|\eta_t|^2\right]<\infty
    $$
\end{lemma}
\begin{proof}
    Fix a stopping time $\tau$. Clearly
    \begin{align}\nonumber
    |\eta_\tau|^2&\leq 4\left(\int_0^{T} |f_s|dA_s\right)^2+4\left(\int_0^{T} |g_s|ds\right)^2
    +4|h_\tau|^2\indb{\tau<T}+4|\xi|^2\\
\label{etaintegrabile}
    &\leq \frac{4}{\beta}\int_0^Te^{\beta A_s}f_s^2dA_s
    +4T\int_0^Te^{\beta A_s}|g_s|^2ds+4\sup\limits_{t\in[0,T]}e^{\beta A_t}|h_t|^2+4 e^{\beta A_T}\xi^2,
    \end{align}
    and since the right-hand side has finite expectation
     we obtain the class $[D]$ property. Likewise, by taking the supremum over all $t\in[0,T]$, and expectation after that, we obtain the second property.
\end{proof}

Now, using the Snell envelope theory, we show that there exists a solution to the equation above. Appendix \ref{app:snell_env} lists the properties that we will need in the following.

\begin{prop}
    \label{prop:given_generators_case}
    Let assumptions \ref{Ass:compensatore}, \ref{Ass:data}-(i)(iv) and \ref{Ass:givengen} hold for some $\beta>0$, then there exists a unique  solution to \eqref{eq:simple_Y}.
\end{prop}

\begin{proof} The uniqueness property is stated and proved separately
in Proposition \ref{uniquenessgivengenerators} below. Existence is proved in several steps.\\
    \textit{Step 1.} We start by defining $Y_t$, for all $t\ge 0$, as the optimal value of the stopping problem:
    \begin{equation}
        \label{eq:Y_def}
    Y_t=\essup\limits_{\tau \in \mathcal{T}_t}\econd{\int_t^{\tau\wedge T} f_sdA_s+\int_t^{\tau\wedge T} g_sds
    +h_\tau\indb{\tau<T}+\xi\indb{\tau\geq T}}.
    \end{equation}
    From \eqref{etaintegrabile} it follows  that $Y_t$ is integrable for all $t$
    and $Y_t=\xi$ for $t\ge T$.
    We have the following a priori estimate on $Y$, that we will prove later.
    \begin{lemma}
        Assume \ref{Ass:data}-(i)(iv) and \ref{Ass:givengen} above on $\xi,f,h,\xi$.
        Then
        \begin{equation}\label{eq:Y_sup_bound}\eval\left[\sup_{t\in[0,T]}e^{\beta A_t}Y_t^2\right]<\infty.\end{equation}
        \label{res:exp_value_Y_bounded_better}
    \end{lemma}
    It follows  that
    $$
    Y_t+\int_0^{t\wedge T}f_sdA_s+\int_0^{t\wedge T}g_sds
    =\essup\limits_{\tau \in \mathcal{T}_t}\econd{\eta_\tau}
    $$
    so $Y_t+\int_0^{t\wedge T}f_sdA_s +\int_0^{t\wedge T}g_sds$
    is the Snell envelope of $\eta$,
    that is the smallest supermartingale such that
    $Y_t+\int_0^{t\wedge T}f_sdA_s+\int_0^{t\wedge T}g_sds\geq\eta_t$. Since $\eta$ is càdlàg,
its Snell envelope $R(\eta)$, and hence $Y$, have a càdlàg modification.
We refer to the appendix for a review of the properties of the Snell
envelope that we will use. Also,
from now on
all supermartingales  that we consider in this proof are assumed to
be càdlàg.
    Also, since $\eta$ satisfies
\eqref{eq:app_peskir_cond} by Lemma \ref{etains2},
       $Y+\int_0^{\cdot\wedge T} f_sdA_s
    + \int_0^{\cdot\wedge T}  g_sds$ is of class $[D]$ and thus
    it admits a unique Doob-Meyer decomposition
    \begin{equation}
    \label{eq:decomposition}
    Y_t+\int_0^{t\wedge T}f_sdA_s\int_0^{t\wedge T}g_sds=M_t-K_t,
    \end{equation}
    where $M$ is a martingale and $K$ is a predictable
     increasing process starting from zero.
From Lemma \ref{etains2} and 
it follows that $\eval K_T^2<\infty$, so that $M$ is a square integrable
martingale.
     Furthermore, $K$ can be decomposed into $K^c+K^d$,
     the continuous and discontinuous part, and we have that
$ \Delta K_t=\Delta K_t\ind_{\{R(\eta)_{t^-}=\eta_{t^-}\}}$
     (see \ref{res:ap_snell_decomp}). However it is immediate to see
     that $R(\eta)_{t^-}=\eta_{t^-}$ if and only if
$Y_{t^-}=h_{t^-}\ind_{\{t\le T\}} + \xi \ind_{\{t> T\}}$ and it follows that
     \begin{equation}\label{eq:K_jumps_snell}
\Delta K_t=\Delta K_t\ind_{\{Y_{t^-}=h_{t^-}\}}, \quad t\in [0,T]   .
    \end{equation}
    By the martingale representation theorem, there exists some $U$ and $Z$ such that
    \begin{gather}\label{rapprmart}
    \evals{\int_0^T\int_E|U_t(e)|\phi_t(de)dA_t}+\evals{\int_0^T|Z_t|^2dt}<\infty\\
    M_t=M_0+\int_0^t\int_EU_s(e)q(dsde)+\int_0^tZ_sdW_s.
    \end{gather}
    Choosing $\tau=t$ in \eqref{eq:Y_def} we see that a.s.
      $Y_t\geq h_t$ for all $t<T$  and $Y_T=\xi$, so $Y_t\geq h_t$ for all $t\leq T$ a.s.
Plugging \eqref{rapprmart} in \eqref{eq:decomposition} we conclude that
the first equality in \eqref{eq:simple_Y} is verified.

 \textit{Step 2.}
      In this step we prove that the Skorohod conditions in \eqref{eq:simple_Y} hold.
From \eqref{separatedjumps} it follows that
$ \Delta K_t \le (-\Delta Y_t)^+$ and, taking into account \eqref{eq:K_jumps_snell}, we obtain
    \begin{equation}
    \label{eq:K_jump_condition}
    \Delta K_t\leq (-\Delta Y_t)^+
\ind_{\lbrace Y_{t^-}=h_{t^-}\rbrace}
    =(Y_{t^-}-Y_t)^+\ind_{\lbrace Y_{t^-}=h_{t^-}\rbrace},
    \end{equation}
    that gives us the second condition.
    Consider now $\tilde{Y}_t=Y_t+\int_0^tf_sdA_s+\int_0^tg_sds+K^d_t=M_t-K_t^c$ and $\tilde{\eta}_t=\eta_t+K^d_t$.
    We claim that $\tilde{Y}_t$ is the Snell envelope of $\tilde{\eta}_t$. Indeed, it is a supermartingale that dominates $\tilde{\eta}_t$. Let $Q_t$ be another supermartingale that dominates $\tilde{\eta}_t$. Then $Q_t-K_t^d$ is still a supermartingale, and dominates $\eta_t$.
    Then, since $Y_t+\int_0^tf_sdA_s+\int_0^tg_sds=R(\eta)_t$, $Q_t\geq \tilde{Y}_t$. Then $\tilde{Y}_t$ is the smallest supermartingale that dominates $\tilde{\eta}_t$, and thus its Snell envelope.
    Next, $Y_t+\int_0^tf_sdA_s+K_t^d=M_t-K_t^c$ is regular
    (we recall that a process $X$ is regular if $X_{t^-}=\leftidx{^p}X_t$,
    where $\leftidx{^p}X_t$ denotes the predictable projection, see also \ref{res:ap_snell_regular};
    all uniformly integrable   càdlàg  martingales are regular). Then, the stopping time defined as
    \begin{equation*}
    D_t^*=\inf\left\lbrace s\geq t : M_s\neq R(\tilde{\eta})_s\right\rbrace=\inf\left\lbrace s\geq t : K^c_s> K^c_t\right\rbrace
    \end{equation*}
    is the largest optimal stopping time, and it satisfies:
    \begin{align*}
    &\tilde{Y}_{D_t^*}=\tilde{\eta}_{D_t^*}\\
    &\tilde{Y}_{s\wedge D_t^*} \text{ is a } \mathbb{F}\text{-martingale}
    \end{align*}
    See (\ref{res:ap_snell_optimcar}).
    Define then
    \begin{equation*}
    D_t=\inf\left\{s\geq t : \tilde{Y}_s\leq \tilde{\eta}_s\right\}
    \end{equation*}
    Since $\tilde{Y}_{D_t^*}=\tilde{\eta}_{D_t^*}$ we have  $D_t\leq D_t^*$, and it follows that
    \begin{equation*}
    0=\int_t^{D_t}\left(\tilde{Y}_s-\tilde{\eta}_s\right)dK^c_s=\int_t^{D_t}\left(Y_s-h_s\right)dK^c_s,
    \end{equation*}
    which implies $K^c_{D_t}=K^c_t$ for   arbitrary $t$, and hence
    $
    \int_0^{T}\left(Y_s-h_s\right)dK^c_s=0$,
    that together with \eqref{eq:K_jump_condition} gives us the Skorohod conditions.

    \textit{Step 3.} We conclude the proof showing that the processes are in the right spaces.
    We have already noticed that $\eval[K_T^2]<\infty$.
Next we define the sequence of stopping times:

    \begin{multline*}
    S_n=\inf\left\lbrace t\in[ 0,T] :\int_0^te^{\beta A_s}|Y_s|^2dA_s+\int_0^t\int_Ee^{\beta A_s}|U_s(e)|^2\phi_s(de)dA_s\right.\\
    \left.+\int_0^te^{\beta A_s}|Z_s|^2ds>n\right\rbrace,
    \end{multline*}
    and consider the ``Ito Formula" applied to $e^{\beta(A_t+t)}Y_t^2$ between $0$ and $S_n$. We have
    \begin{align*}
        e^{\beta(A_{S_n}+S_n)}Y_{S_n}^2&=Y_0^2+\beta\int_{0}^{S_n}e^{\beta(A_s+s)}Y_s^2dA_s+\beta\int_{0}^{S_n}e^{\beta(A_s+s)}Y_s^2ds\\
        &+2\int_{0}^{S_n}\int_Ee^{\beta(A_s+s)}Y_{s^-}U_s(e)q(dsde)+2\int_{0}^{S_n}e^{\beta(A_s+s)}Y_sZ_sdW_s\\
        &-2\int_{0}^{S_n}e^{\beta(A_s+s)}Y_sf_sdA_s-2\int_{0}^{S_n}e^{\beta(A_s+s)}Y_sg_sds\\
        &-2\int_{0}^{S_n}e^{\beta(A_s+s)}Y_{s^-}dK_s+\int_{0}^{S_n}e^{\beta(A_s+s)}Z_s^2ds\\
        &+\sum_{0<s\leq S_n}e^{\beta(A_{s}+s)}\Delta K^2_{s}+\int_{0}^{S_n}\int_E e^{\beta(A_s+s)}U_s^2(e)p(dsde)\\
\end{align*}
Now we use the fact that
$$
\int_{0}^{t}\int_E U_s(e)p(dsde)=\int_{0}^{t}\int_E U_s(e)\phi_s(de)dA_s+\int_{0}^{t}\int_E U_s(e)q(dsde),
$$
and, by Remark \ref{skorohodalt},
$$
\int_0^te^{\beta(A_s+s)} Y_{s^-}dK_s=\underbracket{\int_0^te^{\beta(A_s+s)}(Y_{s^-}-h_{s^-})dK_s}_{=0}+\int_0^te^{\beta(A_s+s)} h_{s^-}dK_s.
$$
Neglecting the positive terms $Y_0^2$ and $\sum_{0<s\leq S_n}e^{\beta(A_{s}+s)}\Delta K^2_{s}$ the previous equation becomes
\begin{align*}
e^{\beta(A_{S_n}+S_n)}Y_{S_n}^2&\geq \beta\int_{0}^{S_n}e^{\beta(A_s+s)}Y_s^2dA_s+\beta\int_{0}^{S_n}e^{\beta(A_s+s)}Y_s^2ds\\
&+2\int_{0}^{S_n}e^{\beta(A_s+s)}Y_{s^-}U_s(e)q(dsde)+2\int_{0}^{S_n}e^{\beta(A_s+s)}Y_sZ_sdW_s\\
&-2\int_{0}^{S_n}e^{\beta(A_s+s)}Y_sf_sdA_s-2\int_{0}^{S_n}e^{\beta(A_s+s)}Y_sg_sds\\&
-2\int_{0}^{S_n}e^{\beta(A_s+s)}h_{s^-}dK_s+\int_{0}^{S_n}\int_E e^{\beta(A_s+s)}U_s^2(e)\phi_s(de)dA_s\\
&+\int_{0}^{S_n}\int_E e^{\beta(A_s+s)}U_s^2(e)q(dsde)+\int_{0}^{S_n}e^{\beta(A_s+s)}Z_s^2ds,
        \end{align*}

    By the definition of $S_n$ and remembering that $Y$ satisfies \eqref{eq:Y_sup_bound}, and using Burkholder-Davis-Gundy inequality we have that
    $$
    \int_0^{t\wedge S_n}e^{\beta(A_s+s)}Y_sZ_sdW_s
    $$
    is a martingale.
    Indeed we have
    \begin{align}\nonumber
    \eval&\left[\sup\limits_{t\in [0,T]}\left|\int_0^{t\wedge S_n}e^{\beta(A_s+s)}Y_sZ_sdW_s\right|\right]
    \leq\eval\left[\left(\int_0^{S_n}e^{2\beta(A_s+s)}Y_s^2Z_s^2ds\right)^{1/2}\right]\\
    \nonumber
    &\leq e^{\beta T}\eval\left[\sup_te^{\beta A_t/2}|Y_t|\left(\int_0^{S_n}e^{\beta A_s}Z_s^2ds\right)^{1/2}\right]\\
    &\leq n^{1/2}e^{\beta T}\eval\left[\sup_te^{\beta A_t}Y_t^2\right]<\infty.
\label{provamartingale}
    \end{align}
Similarly, since
\begin{multline}\label{provamartingaledue}
    \eval\left[\int_0^t\int_Ee^{\beta(A_s+s)} |Y_{s^-}U_s(e)|\phi_s(de)dA_s\right]\leq
\eval\left[\int_0^te^{\beta(A_s+s)} Y_s^2dA_s\right]\\+\eval\left[\int_0^t\int_E e^{\beta(A_s+s)} U_s^2(e)\phi_s(de)dA_s\right]\leq 2n<\infty,
\end{multline}
we obtain that
$\int_0^{t\wedge S_n}\int_Ee^{\beta(A_s+s)} Y_{s^-}U_s(e)q(dsde)$ is a martingale.
Reordering terms and taking expectation we obtain
\begin{align}
\label{eq:general_bound}
\begin{split}
&\beta\eval\left[\int_{0}^{S_n}e^{\beta(A_s+s)}Y_s^2dA_s\right]+\evals{\int_{0}^{S_n}\int_E e^{\beta(A_s+s)}U_s^2(e)\phi_s(de)dA_s}\\
&+\beta\eval\left[\int_{0}^{S_n}e^{\beta(A_s+s)}Y_s^2ds\right]+\evals{\int_{0}^{S_n}e^{\beta(A_s+s)}Z_s^2ds}\\
&\leq \evals{e^{\beta(A_{S_n}+S_n)}Y_{S_n}^2} + 2\evals{\int_{0}^{S_n}e^{\beta(A_s+s)}Y_sf_sdA_s}\\
&+2\evals{\int_{0}^{S_n}e^{\beta(A_s+s)}Y_sg_sds}+ 2\evals{\int_{0}^{S_n}e^{\beta(A_s+s)}h_{s^-}dK_s}\\
&\leq \evals{\sup_t e^{\beta(A_{t}+t)}Y_{t}^2} + \frac{\beta}{2}\evals{\int_{0}^{S_n}e^{\beta(A_s+s)}Y_s^2dA_s} +\frac{\beta}{2}\evals{\int_{0}^{S_n}e^{\beta(A_s+s)}Y_s^2ds}\\
&+ \frac{1}{\beta}\evals{\int_0^Te^{\beta(A_s+s)}f_s^2dA_s} +\frac{2}{\beta}\evals{\int_0^Te^{\beta(A_s+s)}g_s^2ds}
\\
&+\gamma\evals{\sup_t e^{(\beta+\delta)(A_t+t)}h_{t^-}^2}
+\frac{1}{\gamma}\evals{\left(\int_0^{S_n}e^{(\beta-\delta)\frac{A_s+s}{2}}dK_s\right)^2},
\end{split}
\end{align}
where $\gamma>0$ is a constant whose value will be chosen sufficiently large afterwards.
We only need to estimate the last term with the integral in $dK$. In order to do that we apply Ito's formula to $e^{(\beta-\delta)\frac{A_t+t}{2}}Y_t$ between 0 and a stopping time $\tau$, obtaining the following relation
\begin{align*}
 \left(\int_0^{\tau}e^{(\beta-\delta)\frac{A_s+s}{2}}dK_s\right)^2&=\left(Y_0-e^{(\beta-\delta)\frac{A_\tau+\tau}{2}}Y_\tau+\frac{\beta-\delta}{2}\int_0^\tau e^{(\beta-\delta)\frac{A_s+s}{2}}Y_sdA_s\right.\\
 &+\frac{\beta-\delta}{2}\int_0^\tau e^{(\beta-\delta)\frac{A_s+s}{2}}Y_s ds-\int_0^\tau e^{(\beta-\delta)\frac{A_s+s}{2}}f_s dA_s\\
 &-\int_0^\tau e^{(\beta-\delta)\frac{A_s+s}{2}}g_s ds+\int_0^\tau\int_E e^{(\beta-\delta)\frac{A_s+s}{2}}U_s(e) q(dsde)\\
 &\left.+\int_0^\tau e^{(\beta-\delta)\frac{A_s+s}{2}}Z_s dW_s \right)^2
\end{align*}
Notice that the following holds:
\begin{align*}
\left(\int_0^\tau e^{(\beta-\delta)\frac{A_s+s}{2}}Y_sdA_s\right)^2
&\leq \int_0^\tau e^{-\delta (A_s+s)}dA_s\int_0^\tau e^{\beta (A_s+s)}Y_s^2dA_s\\
&\leq\frac{1}{\delta}\int_0^\tau e^{\beta (A_s+s)}Y_s^2dA_s\\
\left(\int_0^\tau e^{(\beta-\delta)\frac{A_s+s}{2}}Y_s ds\right)^2
&\leq \int_0^\tau e^{-\delta A_s}e^{-\delta s}ds\int_0^\tau e^{\beta(A_s+s)}Y_s^2 ds\\
&\leq \frac{1}{\delta}\int_0^\tau e^{\beta(A_s+s)}Y_s^2ds
\end{align*}
and similarly
\begin{align*}
\left(\int_0^\tau e^{(\beta-\delta)\frac{A_s+s}{2}}f_sdA_s\right)^2&\leq \frac{1}{\delta}\int_0^\tau e^{\beta (A_s+s)}f_s^2dA_s\\
\left(\int_0^\tau e^{(\beta-\delta)\frac{A_s+s}{2}}g_sds\right)^2& \leq\frac{1}{\delta}\int_0^\tau e^{\beta (A_s+s)}g_s^2ds
\end{align*}
We note that for a $\mathcal{P}\otimes\mathcal{E}$ measurable process $H$ we have
$$
    \eval\left[\left(\int_0^t\int_E H_s(e)q(dsde)\right)^2\right]\leq
    \eval\left[\int_0^t\int_E H_s^2(e)\phi_s(de)dA_s\right].
$$
This can be checked for instance by applying the Ito formula to compute $N^2_t$ where
$N_t=\int_0^t\int_E H_s(y)q(dsdy)$ and taking expectation after appropriate localization.
Now by taking expectation and using Ito Isometry
we obtain the following bound for $\left(\int_0^{\tau}e^{(\beta-\delta)\frac{A_s+s}{2}}dK_s\right)^2$:
\begin{align*}
 \eval&\left[\left(\int_0^{\tau}e^{(\beta-\delta)\frac{A_s+s}{2}}dK_s\right)^2\right]\leq 16\evals{\sup_t e^{\beta(A_t+t)}Y_t^2} +\frac{8}{\delta}\evals{\int_0^\tau e^{\beta(A_s+s)}g_s^2s}\\
 & +2\frac{(\beta-\delta)^2}{\delta}\evals{\int_0^\tau e^{\beta(A_s+s)}Y_s^2ds}
 +2\frac{(\beta-\delta)^2}{\delta}\evals{\int_0^\tau e^{\beta(A_s+s)}Y_s^2dA_s}
 \\
 &+\frac{8}{\delta}\evals{\int_0^\tau e^{\beta(A_s+s)}f_s^2dA_s}+8\evals{\int_0^\tau e^{\beta (A_s+s)}Z_s^2ds}\\
 &+8\evals{\int_0^\tau\int_E e^{\beta (A_s+s)}U_s^2(e)\phi_s(de)dA_s}.
\end{align*}
By plugging this last estimate into \eqref{eq:general_bound}, by choosing $\alpha$, $\gamma$ such that
$$
\gamma>\max\left(8,4\frac{(\beta-\delta)^2}{\beta\gamma}\right)
$$
we obtain
\begin{align*}
 &\eval\left[\int_{0}^{S_n}e^{\beta(A_s+s)}Y_s^2dA_s\right]+\eval\left[\int_{0}^{S_n}e^{\beta(A_s+s)}Y_s^2ds\right]\\
 &+\evals{\int_{0}^{S_n}\int_E e^{\beta(A_s)}U_s^2(e)\phi_s(de)dA_s}+\evals{\int_{0}^{S_n}e^{\beta(A_s)}Z_s^2ds}\\
 & \leq C\left( \evals{\sup_t e^{\beta A_{t}}Y_t^2}+2\left(\frac{1}{\beta} +\frac{1}{\delta\gamma}\right)\evals{\int_0^Te^{\beta A_s}f_s^2dA_s}\right.\\
 &\left.+\evals{\int_0^Te^{\beta A_s}g_s^2ds}+\gamma\evals{\sup_t e^{(\beta+\delta)A_t}h_{t^-}^2}\right),
\end{align*}
for some constant $C$ independent of $n$.
Now let $S=\lim_n S_n$ and by the last estimate, considering how $S_n$ are defined, we have $S=T$.
This implies that $Y\in L^{2,\beta}(A)\cap L^{2,\beta}(W)$, $Z\in L^{2,\beta}(W)$ and $U\in L^{2,\beta}(p)$.

\end{proof}

\begin{proof}[Proof of lemma \ref{res:exp_value_Y_bounded_better}]
    By the definition of $Y$ we have
    \begin{multline*}
    e^{\beta A_t/2}|Y_t|\leq \eval\left[ e^{\beta A_T/2}|\xi|+e^{\beta A_t/2}\int_t^T|f_s|dA_s\right.\\
    +e^{\beta A_t/2}\int_t^T|g_s|ds+\sup\limits_{0\leq s\leq T}e^{\beta A_s/2}|h_s|\left|\vphantom{\int_t^T}\mathcal{F}_t\right]
    \end{multline*}
    Proceeding as in Remark \ref{stimessus} we have
    \begin{equation*}
    \int_t^T|f_s|dA_s\leq \frac{e^{-\beta A_t/2}}{\beta^{1/2}}\left(\int_t^Te^{\beta A_s}|f_s|^2dA_s\right)^{1/2}
    \end{equation*}
    and it follows that
    \begin{multline*}
    e^{\beta A_t/2}|Y_t|\leq \eval \left[e^{\beta A_T/2}|\xi|+\frac{1}{\beta^{1/2}}
    \left(\int_0^Te^{\beta A_s}|f_s|^2dA_s\right)^{1/2}\right. \\
    +\int_0^Te^{\beta A_s/2}|g_s|ds
    +\sup\limits_{0\leq s\leq T}e^{\beta A_s/2}|h_s|\left|\vphantom{\int_t^T}\mathcal{F}_t\right]\eqqcolon S_t
    \end{multline*}
    Under assumption \ref{Ass:data}-(i)(iv) and \ref{Ass:givengen}, $S$ is a square integrable martingale.
Then by Doob's martingale inequality
    $\eval\left[\sup\limits_{0\leq t \leq T}e^{\beta A_t}|Y_t|^2\right]
    \leq C\eval\left[S_T^2\right]<\infty.
$
\end{proof}
\begin{rem}
    Contrary to the diffusive (or diffusive and Poisson) case, the fact that
    $\eval\left[\sup_{t\in[0,T]}e^{\beta A_t}Y_t^2\right]<\infty$ does not imply that $Y\in L^{2,\beta}(A)$. For this to happen we would need additional conditions on $A$, for example $\eval[A_T^2]<\infty$.
\end{rem}


Next we prove   uniqueness.

    \begin{prop}\label{uniquenessgivengenerators}
    Let assumptions \ref{Ass:compensatore}, \ref{Ass:data}-(i)(iv) and \ref{Ass:givengen} hold for some $\beta>0$,
    then the    solution to \eqref{eq:simple_Y} is  unique.
    \label{prop:unique_simple}
    \end{prop}
    \begin{proof}
        Let $(Y',U',Z',K')$ and $(Y'',U'',Z'',K'')$ be two solutions. Define
        $$
        \bar{Y}=Y'-Y''\quad \bar{U}=U'-U'' \quad \bar{Z}=Z'-Z'' \quad \bar{K}=K'-K'',
        $$
        then $(\bar{Y},\bar{U},\bar{Z},\bar{K})$ satisfies
        \begin{equation}
        \label{eq:diff_sols}
        \bar{Y}_t=-\int_t^T\int_E \bar{U}(e)q(dsde)-\int_t^TZ_sdW_s+\bar{K}_T-\bar{K}_t.
        \end{equation}

        We compute $d(e^{\beta (A_t+t)}\bar{Y}^2_t)
        $ by the Ito formula and we obtain
        \begin{multline}
        \label{eq:ito_barY}
        -\bar{Y}^2_0=\beta\int_0^T e^{\beta (A_s+s)}\bar{Y}^2_sdA_s+\beta\int_0^T e^{\beta (A_s+s)}\bar{Y}^2_sds-2\int_0^T\bar{Y}_{s^-}d\bar{K}_s\\+2\int_0^T\int_E e^{\beta (A_s+s)} \bar{Y}_{s^-}U_s(y)q(dsdy)+\int_0^Te^{\beta(A_s+s)}Y_sZ_sdW_s\\
        +\int_0^Te^{\beta(A_s+s)}Z_s^2ds+\sum\limits_{0<s\leq T}e^{\beta(A_s+s)}(\Delta\bar{Y}_s)^2
        \end{multline}
        The last term can be divided in totally inaccessible jumps (from the martingale in $q(dsde)$) and predictable jumps, from the $K$ process, thus:
        \begin{align*}
        \sum\limits_{0<s\leq T}e^{\beta(A_s+s)}(\Delta\bar{Y}_s)^2
        &\geq \sum\limits_{0<T_n\leq T}e^{\beta(A_s+s)}U_{T_n}^2(\xi_n)=\int_0^T\int_EU_s^2(e)p(dsde)\\
        &=\int_0^T\int_E U_s^2(e)q(dsde)+\int_0^T\int_E U_s^2(e)\phi_s(de)dA_s
        \end{align*}
Proceeding as in \eqref{provamartingale} and \eqref{provamartingaledue} we prove that
the stochastic integrals with respect to $W$ and $q$ are martingales.
        By neglecting $Y_0^2$ and taking expectation in \eqref{eq:ito_barY}, we obtain
        \begin{multline*}
        \beta\eval\left[\int_0^Te^{\beta(A_s+s)}\bar{Y}^2_sdA_s\right]+\beta\eval\left[\int_0^Te^{\beta(A_s+s)}\bar{Y}^2_sds\right]\\+\eval\left[\int_0^T\int_Ee^{\beta(A_s+s)}\bar{U}^2_s(y)\phi_s(dy)dA_s\right]+\evals{\int_0^Te^{\beta(A_s+s)}Z_s^2ds}\\\leq 2\eval\left[\int_0^Te^{\beta(A_s+s)}\bar{Y}_{s^-}d\bar{K}_s\right].
        \end{multline*}
        Now, taking into account Remark \ref{skorohodalt} we have
        \begin{align*}
        \int_0^T\bar{Y}_{s^-}d\bar{K}_s
        &=\underbracket{\int_0^T(Y'_{s^-}-h_{s^-})dK'_s}_{=0}-\underbracket{\int_0^T(Y'_{s^-}-h_{s^-})dK''_s}_{\geq 0}+\\
        &\quad-\underbracket{\int_0^T(
        Y''_{s^-}-h_{s^-})dK'_s}_{\geq 0}+\underbracket{\int_0^T(
        Y''_{s^-}-h_{s^-})dK''_s}_{=0}\\
        &\leq 0,
        \end{align*}
    and thus
    \begin{equation*}
    \beta||\bar Y||^2_{L^{2,\beta}(A)}+\beta||\bar Y||^2_{L^{2,\beta}(W)}
    +||\bar U||^2_{L^{2,\beta}(p)}+||\bar Z||^2_{L^{2,\beta}(W)}\leq 0,
    \end{equation*}
    which gives the uniqueness of $Y$, $U$ and $Z$. From \eqref{eq:diff_sols} we obtain
    $$
    \bar{K}_T=\bar{K}_t \quad \forall t \in[0,T].
    $$
    Then $\bar{K}_T=0$ since $\bar{K}_0=0$ and consequently $\bar{K}_t=0$ for all $t$.
    \end{proof}


Consider now the optimal stopping problem with running gains $f,g$, early stopping reward $h$ and non stopping reward $\xi$. This means we are interested in the quantity
\begin{align*}
v(t)&=\essup\limits_{\tau \in \mathcal{T}_t}\econd{\int_t^\tau f_sdA_s+\int_0^\tau g_sds+h_\tau\ind_{\lbrace\tau<T\rbrace}+\xi\ind_{\lbrace\tau\geq T\rbrace}}.
\end{align*}
Notice that we have two running gains, $f$ integrated with respect to the process $A$, and $g$ integrated with respect to Lebesgue measure in time. This could be used for example if we want to describe two different time dynamics, one depending on the speed of the point process.\\
It is possible to show that  the solution to the RBSDE solves the optimal stopping problem and it is possible to identify an $\epsilon$-optimal stopping time. Under additional assumptions, it is possible to find an optimal stopping time. For this we need a definition, given in \cite{kobylanski2012optimal} for admissible families over stopping times, that we adapt to our simpler case:

\begin{defi}
    We say that a process $\phi$ is left \textit{ (resp. right)} upper semi-continuous over stopping times in expectation (USCE) if for all $\theta\in\mathcal{T}_0$, $\eval\left[\phi_\theta\right]<\infty$ and for all sequences of stopping times $(\theta_n)$ such that $\theta_n\uparrow\theta$ \textit{(resp. $\theta_n\downarrow\theta$)} it holds that
    $$
    \eval[\phi_\theta]\geq\limsup\limits_{n\rightarrow\infty}\eval[\phi_{\theta_n}].
    $$
\end{defi}

\begin{rem}
    If $\phi$ is a left upper semi continuous progressive process, then $\phi$ is left upper semi continuous along stopping times. If also $\eval[\sup_{t}|\phi_t|]$ holds, then it is left USCE. Indeed we have
    \begin{align*}
    \limsup\limits_{n\rightarrow\infty}\eval\left[\phi_{\theta_n}\right]\leq
    \eval\left[\limsup\limits_{n\rightarrow\infty}\phi_{\theta_n}\right]\leq\eval\left[\phi_\theta\right].
    \end{align*}
    by using Reverse Fatou's lemma with $\sup_{t}|\phi_t|$ as dominant.
    \end{rem}

\begin{prop}
    \label{prop:representation}
    Let assumptions \ref{Ass:compensatore}, \ref{Ass:data}-(i)(iv) and \ref{Ass:givengen} hold. Then we have:
    \begin{enumerate}[leftmargin=*]
        \item  The solution to the RBSDE \eqref{eq:simple_Y} is a solution to the optimal stopping problem

        $$
        Y_t=\essup\limits_{\tau \in \mathcal{T}_t}\econd{\int_t^\tau f_sdA_s+\int_0^\tau g_sds+h_\tau\ind_{\lbrace\tau<T\rbrace}+\xi\ind_{\lbrace\tau\geq T\rbrace}}.
        $$
        \item For all $\epsilon>0$, define $D_t^\epsilon$ as
        $$
        D_t^\epsilon=\inf\left\lbrace s\geq t : Y_s\leq h_s+\epsilon \right\rbrace \wedge T.
        $$
        Then $D_t^\epsilon$ is an $\epsilon$-optimal stopping time in the sense that
        $$
        Y_t\leq \essup\limits_{\tau \in \mathcal{T}_t}\econd{\int_t^{D_t^\epsilon} f_sdA_s+\int_0^{D_t^\epsilon} g_sds+h_{D_t^\epsilon}\ind_{\lbrace{D_t^\epsilon}<T\rbrace}+\xi\ind_{\lbrace{D_t^\epsilon}\geq T\rbrace}}+\epsilon.
        $$
        \item
        If in addition  $h_t\ind_{\lbrace t<T\rbrace}+\xi\ind_{\lbrace t\geq T\rbrace}$ is left USCE, then
        $$
        \tau_t^*=\inf\left\lbrace s\geq t : Y_s\leq h_s \right\rbrace \wedge T.
        $$
        is optimal and is the smallest of all optimal stopping times.
    \end{enumerate}

\end{prop}
\begin{rem}
    The condition on the third point may seem unusual, but it is satisfied for example if $h_t$ is left upper semi continuous on [0,T] and $h_T<\xi$.
\end{rem}
\begin{proof}
    Let $\tau\in\mathcal{T}_t$ and consider the first equation \eqref{eq:simple_Y} between $t$ and $\tau$:
    $$
    Y_t=Y_\tau+\int_t^\tau f_sdA_s+\int_t^\tau g_sds-\int_t^\tau \int_EZ_s(y)q(dsdy)+K_\tau-K_t.
    $$
    By taking conditioning at $\mathcal{F}_t$ we have
    \begin{align}
    Y_t&=\econd{Y_\tau+\int_t^\tau f_sdA_s+\int_t^\tau g_sds+K_\tau-K_t}\label{eq:cond_equation}\\
    &\geq\econd{h_\tau\ind_{\lbrace\tau<T\rbrace}+\xi\ind_{\lbrace\tau\geq T\rbrace}+\int_t^\tau f_sdA_s+\int_t^\tau g_sds}\label{eq:cond_superior},
    \end{align}
    since the integral on $q$ is a martingale, $K$ is increasing and $Y_t\geq h_\tau\ind_{\lbrace t<T\rbrace}+\xi\ind_{\lbrace t=T\rbrace}$.
    To prove the reverse inequality, consider $\epsilon>0$ and the corresponding $D_t^\epsilon$. It holds that $Y_{D_t^\epsilon}\leq h_{D_t^\epsilon}+\epsilon$ on $ \{D_t^\epsilon<T\} $. And on $\{D_t^\epsilon=T\}$ we have $Y_u>h_u+\epsilon$ for all $t\leq u < T$. Then , between $t$ and $D_t^\epsilon$, $Y_{s^-}>h_{s^-}$ and thus

    $$
    \int_t^{D_t^\epsilon}(Y_{s^-}-h_{s^-})dK_s=0 \Rightarrow K_{D_t^\epsilon}=K_t.
    $$
    Considering all this in \eqref{eq:cond_equation} we have

    \begin{align}
    Y_t&=\econd{Y_{D_t^\epsilon}+\int_t^{D_t^\epsilon} f_sdA_s+\int_{t}^{D_t^\epsilon}g_sds}\nonumber\\
    &\leq \econd{h_{D_t^\epsilon}\ind_{\lbrace{D_t^\epsilon}<T\rbrace}+\xi\ind_{\lbrace{D_t^\epsilon}=T\rbrace}+\int_t^{D_t^\epsilon} f_sdA_s+\int_{t}^{D_t^\epsilon}g_sds}+\epsilon. \label{eq:epsilon_opt}
    \end{align}
    This together with \eqref{eq:cond_superior} proves points one and two.
    For the third point, notice that $D_t^\epsilon$ are non increasing in $\epsilon$ and that $D_t^\epsilon\leq\tau^*$. Thus $D_t^\epsilon\rightarrow D_t^0\leq\tau^*$ when $\epsilon\rightarrow 0$.
    Now since $h_t\ind_{\lbrace t<T\rbrace}+\xi\ind_{\lbrace t=T\rbrace}$ is left USCE and the integral part is too, we have from \eqref{eq:epsilon_opt}
    \begin{align*}
    \eval[Y_t]\leq\limsup\limits_{\epsilon\rightarrow 0}\eval\left[h_{D_t^\epsilon}\ind_{\lbrace{D_t^\epsilon}<T\rbrace}+\xi\ind_{\lbrace{D_t^\epsilon}=T\rbrace}+\int_t^{D_t^\epsilon} f_sdA_s+\int_t^{D_t^\epsilon} g_sds\right]\\
    \leq \eval\left[h_{D_t^0}\ind_{\lbrace{D_t^0}<T\rbrace}+\xi\ind_{\lbrace{D_t^0}=T\rbrace}+\int_t^{D_t^0} f_sdA_s+\int_t^{D_t^0} g_sds\right].
    \end{align*}
    Thus we have
    $$
    \eval\left[Y_t\right]= \eval\left[h_{D_t^0}\ind_{\lbrace{D_t^0}<T\rbrace}+\xi\ind_{\lbrace{D_t^0}=T\rbrace}+\int_t^{D_t^0} f_sdA_s+\int_t^{D_t^0} g_sds\right],
    $$
    so  $D_t^0$ is optimal (see \ref{res:ap_snell_optimcar}). We only need to prove that $D_t^0=\tau^*$. We already know that $D_t^0\leq\tau^*$. On the other hand, since $D_t^0$ is optimal it holds that $Y_{D_t^0}=\eta_{D_t^0}$, and thus by the definition of $\tau^*$, $\tau^*\leq D_t^0$. This also proves that $\tau^*$ is the smallest optimal stopping time.
    \end{proof}

A further interesting property holds when the reward is left USCE:
\begin{prop}
Under assumptions  \ref{Ass:data}-(i)(iv) and \ref{Ass:givengen}, if $h_\tau\ind_{\lbrace\tau<T\rbrace}+\xi\ind_{\lbrace\tau\geq T\rbrace}$ is also left USCE, then $K$ in the solution of $\eqref{eq:simple_Y}$ is continuous.
\end{prop}

\begin{proof}
The proof is given in \cite{kobylanski2012optimal} in the case were the reward is a positive progressive process $\phi$ of class [D].
We can adapt to our case by using the transformation
$$
I=\inf_t\eta_t \qquad N_t=\econd{I} \qquad \tilde{\eta}_t=\eta_t-N_t.
$$
We have that $\tilde{\eta}_t$ is USCE, as $\eval[\tilde{\eta}_t]=\eval[\eta_t]-\eval[I]$. Indeed let $\theta_n\uparrow\theta$, then
$$
\limsup\limits_{n\rightarrow\infty}\eval[\tilde{\eta}_{\theta_n}]\leq\eval[\tilde{\eta}_\theta].
$$
Then if $R(\eta)$ denotes the Snell envelope of $\eta$, it holds that
$   R(\tilde{\eta})=R(\eta)-N_t$. The Doob-Meyer decomposition for the càdlàg supermartingale $R(\tilde{\eta})$ holds:
$$
R(\tilde{\eta})_t=\tilde{M}_t-\bar{K}_t
$$
With $\bar{K}$ continuous thanks to Proposition B.10 in \cite{kobylanski2012optimal}.
Then $Y_t+\int_0^tf_sdA_s=R(\eta)=R(\tilde{\eta})+N_t=\tilde{M}_t+N_t-\bar{K}_t$, but since the decomposition is unique, $\int_0^t\int_E Z_s(y)q(dsdy)=M_t=\tilde{M}+N_t$ and $K_t=\bar{K}_t$. Thus the term $K$ is continuous.
\end{proof}


If we are interested only in \eqref{eq:system_Y_only_MPP}, and we have a filtration generated only by a MPP and $g\equiv0$, the proofs above are still applicable. In this case, there is no particular reason to use a $L^2$ space, since the martingale representation theorem for marked point processes works in $L^1$ (see \cite{jacod1975multivariate}). We thus obtain the following:
\begin{prop}
    \label{prop:only_MPP_simple}
    Let assumption \ref{Ass:compensatore} hold. Let ${\xi}$ be a $\mathcal{G}_T$-measurable random variable. Let ${f},{h}$ be  $\mathbb{G}$-progressive processes. Assume that
    $$
    \evals{|{\xi}|+\int_0^T|{f}_s|dA_s+\sup\limits_{t\in[0,T]}|{h}_t|}<\infty.
    $$
    Then there exists a unique solution to the system
    \begin{equation}
    \label{eq:simple_Y_only_MPP}
    \begin{cases}
    {Y}_t={\xi}+\int_t^T{f}_sdA_s-\int_t^T\int_E{U}_s(y)q(dsdy)+{K}_T-{K}_t\\
    {Y}_t\geq {h}_t \qquad \forall t\in[0,T] \text{ a.s.}\\
    \int_0^T({Y}_{s}-{h}_{s})d{K}^c_s=0 \text{ and } \Delta {K}_s\leq({h}_{s^-}-{Y}_s)^+\ind_{\lbrace {Y}_{s^-}={h}_{s^-}\rbrace}.
    \end{cases}
    \end{equation}
    where ${Y}$ is a càdlàg $\mathbb{G}$-adapted process such that $\eval\left[|{Y}_t|\right]<\infty$ for all $t$, ${K}$ is a $\mathbb{G}$-predictable càdlàg increasing process with ${K}_0=0$ and $\eval\left[{K}_T\right]<\infty$ and ${U}$ is a $\mathcal{P}(\mathbb{G})\otimes\mathcal{E}$-measurable process such that
    $\eval\left[\int_0^T\int_E|U_s(e)|\phi_s(de)dA_s\right]<\infty.$
\end{prop}

\begin{proof}
    Existence of a solution is obtained as in \ref{prop:given_generators_case}. The process $\eta_t$ satisfies then the weaker condition $\evals{\sup_t|\eta_t|}<\infty$, but this is enough to apply the Snell's envelope results (see appendix \ref{app:snell_env}, in particular \eqref{eq:app_peskir_cond}). Integrability is straightforward. Now let $(Y',U',K')$ and $(Y'',U'',K'')$ be two solutions, their difference satisfies
    \begin{equation}
    \label{eq:difference_only_MPP}
    Y'_t-Y''_t=Y'_0-Y''_0+\int_0^t\int_E(U'_s(e)-U''_s(e))q(dsde)-(K'_t-K''_t).
    \end{equation}
    Uniqueness of the component ${Y}$ comes from the fact that if $({Y},{U},{K})$ satisfies the equation, the càdlàg process ${Y}$ satisfies
    $$
    {Y}_t=\essup\limits_{\tau \in \mathcal{T}_t}\left[\left.\int_t^{\tau\wedge T}{f}_sdA_s+{h}_\tau\indb{\tau<T}+{\xi}\indb{\tau\geq T}\right|\mathcal{G}_t\right],
    $$
    which can be shown as in proposition \ref{prop:representation}, adapted to the this case with less integrability. Relation \eqref{eq:difference_only_MPP} becomes
    $$
    \int_0^t\int_EU'_s(e)q(dsde)-K'_t=\int_0^t\int_EU''_s(e)q(dsde)-K''_t.
    $$
    Since the predictable jumps of $K$ and the totally inaccessible jumps of the integrals with respect to $q$ are disjoint, we have that $U'_{T_n}(\xi_n)=U''_{T_n}(\xi_n)$ for all $n$. Then
    \begin{align*}
    \int_0^T\int_E|U'_s(e)-U''_s(e)|\phi_s(de)dA_s&=\int_0^T\int_E|U'_s(e)-U''_s(e)|p(dsde)\\
    &=\sum_{n\geq 1}^{}|U'_{T_n}(\xi_n)-U''_{T_n}(\xi_n)|=0,
    \end{align*}
    and thus $U'_s(e)=U''_s(e)$ $\phi_s(de)dA_sd\mathbb{P}$-a.e. Then $K'_t=K''_t$ a.s. and uniqueness is proven.
\end{proof}

We have then a result for optimal stopping analogous to proposition \ref{prop:representation}:
\begin{prop}
Assume that the conditions of proposition \ref{prop:only_MPP_simple} hold. Then
    \begin{enumerate}[leftmargin=*]
	\item  The solution to the RBSDE \eqref{eq:simple_Y_only_MPP} is a solution to the optimal stopping problem
	
	$$
	Y_t=\essup\limits_{\tau \in \mathcal{T}_t}\econd{\int_t^\tau f_sdA_s+\int_0^\tau g_sds+h_\tau\ind_{\lbrace\tau<T\rbrace}+\xi\ind_{\lbrace\tau\geq T\rbrace}}.
	$$
	\item For all $\epsilon>0$, define $D_t^\epsilon$ as
	$$
	D_t^\epsilon=\inf\left\lbrace s\geq t : Y_s\leq h_s+\epsilon \right\rbrace \wedge T.
	$$
	Then $D_t^\epsilon$ is an $\epsilon$-optimal stopping time in the sense that
	$$
	Y_t\leq \essup\limits_{\tau \in \mathcal{T}_t}\econd{\int_t^{D_t^\epsilon} f_sdA_s+\int_0^{D_t^\epsilon} g_sds+h_{D_t^\epsilon}\ind_{\lbrace{D_t^\epsilon}<T\rbrace}+\xi\ind_{\lbrace{D_t^\epsilon}\geq T\rbrace}}+\epsilon.
	$$
	\item
	If in addition  $h_t\ind_{\lbrace t<T\rbrace}+\xi\ind_{\lbrace t\geq T\rbrace}$ is left USCE, then
	$$
	\tau_t^*=\inf\left\lbrace s\geq t : Y_s\leq h_s \right\rbrace \wedge T.
	$$
	is optimal and is the smallest of all optimal stopping times. Moreover, the process $K$ is continuous.
\end{enumerate}
\end{prop}

    \section{Reflected BSDE}
    \label{sec:general_BSDE}
        We now turn to the case where the generators depend on the solution, that is equation \eqref{eq:system_Y_both}.
%
Denote by $\lambda$ the Lebesgue measure on $[0,T]$, and introduce now $L^{2,\beta}(\Omega\times[0,T],\mathcal{F}\otimes\mathcal{B}([0,T]),(A(\omega,dt)+\lambda(dt))$, the space of all $\mathbb{F}$-progressive processes such that
$$
\|Y\|^2_{L^{2,\beta}(A+\lambda)}=\evals{\int_{0}^{T}e^{\beta A_s}Y_s^2(dA_s+ds)}<\infty.
$$
For brevity we denote is as $L^{2,\beta}(A+\lambda)$ in the following. It is a Hilbert space equipped with the norm above. It is clear that a process is in $L^{2,\beta}(A+\lambda)$ if and only if lies in $Y\in L^{2,\beta}(A)\cap L^{2,\beta}(W)$.

\begin{teo}
    \label{teo:both_processes}
    Let assumption \ref{Ass:compensatore} and \ref{Ass:data} hold for some $\beta>L_p^2+2L_f$. Then there exists a unique solution to \eqref{eq:system_Y_both}.
\end{teo}

\begin{proof}

    We con We will use a contraction theorem on  $$\mathbb{L}^\beta=L^{2,\beta}(\Omega\times[0,T],\mathcal{F}\otimes\mathcal{B}([0,T]),(A(\omega,dt)+\lambda(dt))\mathbb{P}(d\omega))\times L^{2,\beta}(p)\times L^{2,\beta}(W).$$

    We construct a mapping $\Gamma$ that to each $(P,Q,R)\in L^{2,\beta}(A+\lambda)\times L^{2,\beta}(p)\times L^{2,\beta}(W)$ associates $(Y,U,Z)$ solution to equation \eqref{eq:simple_Y} when the generators are given by $f_t(P_t,Q_t)$ and $g_t(P_t,R_t)$.
    Such map is well defined: indeed if we fix $(P,Q,R)\in L^{2,\beta}(A+\lambda)\times L^{2,\beta}(p)\times L^{2,\beta}(W)$, thanks to assumption \ref{Ass:data}, the generators are known process that satisfy assumption \ref{Ass:givengen} and proposition \ref{prop:given_generators_case} and \ref{prop:unique_simple} give us the existence and uniqueness of $(Y,U,Z)\in L^{2,\beta}(A+\lambda)\times L^{2,\beta}(p)\times L^{2,\beta}(W)$. Notice that thanks to the Lipschitz conditions on $g$ and $f$, if we take two triplets $(P',Q',R')\equiv(P'',Q'',R'')$ in $L^{2,\beta}(A+\lambda)\times L^{2,\beta}(p)\times L^{2,\beta}(W)$, then $f_s(Y',U')\equiv f_s(Y'',U'')$ in $L^{2,\beta}(A)$ and $g_s(Y'.Z')\equiv g_s(Y'',Z'')$ in $L^{2,\beta}(W)$.

     Consider now $(P',Q',R')$ and $(P'',Q'',R'')$ in $\mathbb{L}^\beta$, and consider their images through $\Gamma$, $(Y',U',Z')=\Gamma(P',Q',R')$ and $(Y'',U'',Z'')=\Gamma(P'',Q'',R'')$. Denote $\bar{Y}=Y'-Y''$, $\bar{P}=P'-P''$ and so on. Denote also $\bar{f}_t=f_t(P'_t,Q'_t)-f_t(P''_t,Q''_t)$ and similarly denote $\bar{g}$. $(\bar{Y},\bar{U},\bar{Z},\bar{K})$ satisfies
    $$
    \bar{Y}=\int_t^T\bar{f}_sdA_s+\int_t^T\bar{g}_sds-\int_t^T\int_E\bar{U}_s(e)q(dsde)-\int_t^T\bar{Z}_sdW_s+\bar{K}_T-\bar{K}_t.
    $$

    We now apply Ito's Lemma to $e^{\beta A_s}e^{\gamma s}\bar{Y}^2_s$ obtaining, after taking expectation,
    \begin{multline*}
    \beta\evals{\int_0^Te^{\beta A_s}e^{\gamma s}\bar{Y}^2_sdA_s}
    +\gamma\evals{\int_0^Te^{\beta A_s}e^{\gamma s}\bar{Y}^2_sds}
    +\evals{\int_0^Te^{\beta A_s}e^{\gamma s}\bar{Z}^2_sdW_s}\\
    +\evals{\int_0^T\int_Ee^{\beta A_s}e^{\gamma s}\bar{U}^2_s\phi_s(de)dA_s}
    \leq 2\evals{\int_0^Te^{\beta A_s}e^{\gamma s}\bar{f}_s^2dA_s}\\
    +2\evals{\int_0^Te^{\beta A_s}e^{\gamma s}\bar{g}_s^2ds}
    +2\evals{\int_0^T\bar{Y}_{s^-}d\bar{K}_s}.
    \end{multline*}
    As in the proof of proposition \ref{prop:unique_simple}, we have that
    $$
    \int_0^T\bar{Y}_{s^-}d\bar{K}_s\leq 0.
    $$
    Denote by $||\cdot||_{\beta,\gamma,A}$ the norm (equivalent to $||\cdot||_{L^{2,\beta}(A)}$)
    $$
    \left(\evals{\int_0^Te^{\beta A_s}e^{\gamma s}\bar{Y}^2_sdA_s}\right)^{1/2},
    $$
    and similarly denote the norms $||\cdot||_{\beta,\gamma,p}$ and $||\cdot||_{\beta,\gamma,W}$.
    Using the Lipschitz properties of $f$ and $g$ this gives

    \begin{multline*}
    \beta||\bar{Y}||^2_{\beta,\gamma,A}+\gamma||\bar{Y}||^2_{\beta,\gamma,W}+||\bar{U}||^2_{\beta,\gamma,p}+||\bar{Z}||^2_{\beta,\gamma,W}\leq \\\leq 2L_f\evals{\int_0^T e^{\beta A_s}e^{\gamma s}|\bar{Y}_s||\bar{P}_s|dA_s}+2L_p\evals{\int_0^T e^{\beta A_s}e^{\gamma s}|\bar{Y}_s|\left(\int_E|\bar{Q}^2_s|\right)^{1/2}dA_s}\\+2L_g\evals{\int_0^T e^{\beta A_s}e^{\gamma s}|\bar{Y}_s||\bar{P}_s|ds}+2L_W\evals{\int_0^T e^{\beta A_s}e^{\gamma s}|\bar{Z}_s||\bar{R}_s|ds}.
    \end{multline*}
    Using the inequality $2ab\leq \alpha a^2+b^2/\alpha$ for $a,b\geq 0$ we obtain:


        \begin{multline*}
    \beta||\bar{Y}||^2_{\beta,\gamma,A}+\gamma||\bar{Y}||^2_{\beta,\gamma,W}+||\bar{U}||^2_{\beta,\gamma,p}+||\bar{Z}||^2_{\beta,\gamma,W}\\\leq \frac{L_f}{\sqrt{\alpha_p}}||\bar{Y}||^2_{\beta,\gamma,A}+{L_f}{\sqrt{\alpha}}||\bar{P}||^2_{\beta,\gamma,A}+\frac{L_p^2}{\alpha}||\bar{Y}||^2_{\beta,\gamma,A}+\alpha||\bar{Q}||^2_{\beta,\gamma,p}
    \\+\frac{L_g}{\sqrt{\alpha}}||\bar{Y}||^2_{\beta,\gamma,W}+{L_g}{\sqrt{\alpha}}||\bar{P}||^2_{\beta,\gamma,W}+\frac{L_W^2}{\alpha}||\bar{Y}||^2_{\beta,\gamma,A}+\alpha||\bar{R}||^2_{\beta,\gamma,A}.
    \end{multline*}
    Rewriting we obtain the following relation:
    \begin{multline}
    \label{eq:contraction_relation}
    ||\bar{U}||^2_{\beta,\gamma,p}+||\bar{Z}||^2_{\beta,\gamma,W}+\left(\beta-\frac{L_p^2}{\alpha}-\frac{L_f}{\sqrt{\alpha}}\right)||\bar{Y}||^2_{\beta,\gamma,A}\\
    +\left(\gamma-\frac{L_W^2}{\alpha}-\frac{L_g}{\sqrt{\alpha}}\right)||\bar{Y}||^2_{\beta,\gamma,W}
    \\
    \leq {L_f}{\sqrt{\alpha}}||\bar{P}||^2_{\beta,\gamma,A}+
    \alpha||\bar{Q}||^2_{\beta,\gamma,p}+
    {L_g}{\sqrt{\alpha}}||\bar{P}||^2_{\beta,\gamma,W}+
    \alpha||\bar{R}||^2_{\beta,\gamma,A}.
    \end{multline}
    Since $\beta>L_p^2+2L_f$, it is possible to choose $\alpha\in(0,1)$ such that
    $$
    \beta>\frac{L^2_p}{\alpha}+\frac{2L_f}{\sqrt{\alpha}},
    $$
    and for that $\alpha$, choose $\gamma$ such that $\gamma>L^2_W/\alpha+2L_g/\sqrt{\alpha}$. The relation \eqref{eq:contraction_relation} rewrites as
    \begin{multline}
    \label{eq:contraction_relation_more}
\frac{L_f}{\sqrt{\alpha}}||\bar{Y}||^2_{\beta,\gamma,A}+\frac{L_g}{\sqrt{\alpha}}||\bar{Y}||^2_{\beta,\gamma,W}+||\bar{U}||^2_{\beta,\gamma,p}+||\bar{Z}||^2_{\beta,\gamma,W}\\
    \leq {L_f}{\sqrt{\alpha}}||\bar{P}||^2_{\beta,\gamma,A}+
    \alpha||\bar{Q}||^2_{\beta,\gamma,p}+
    {L_g}{\sqrt{\alpha}}||\bar{P}||^2_{\beta,\gamma,W}+
    \alpha||\bar{R}||^2_{\beta,\gamma,A}\\
    =\alpha\left(\frac{L_f}{\sqrt{\alpha}}||\bar{P}||^2_{\beta,\gamma,A}+
    ||\bar{Q}||^2_{\beta,\gamma,p}+
    \frac{L_g}{\sqrt{\alpha}}||\bar{P}||^2_{\beta,\gamma,W}+
    ||\bar{R}||^2_{\beta,\gamma,A}\right).
    \end{multline}

    Now
    $$
    \frac{L_f}{\sqrt{\alpha}}||\bar{P}||^2_{\beta,\gamma,A}+\frac{L_g}{\sqrt{\alpha}}||\bar{P}||^2_{\beta,\gamma,W}=\evals{\int_0^Te^{\beta A_s}e^{\gamma s}\bar{P}_s^2(\frac{L_f}{\sqrt{\alpha}}dA_s+\frac{L_g}{\sqrt{\alpha}}ds)}
    $$
    is a norm equivalent to $\|\bar{P}\|_{L^{2,\beta}(A+\lambda)}$. We have thus that $\Gamma$ is a contraction on $\mathbb{L}$ for the equivalent norm
    $$
    \|(Y,U,Z)\|^2_{\mathbb{L}^\beta,\gamma}=\frac{L_f}{\sqrt{\alpha}}||Y||^2_{\beta,\gamma,A}+\frac{L_g}{\sqrt{\alpha}}||Y||^2_{\beta,\gamma,W}+||\bar{U}||^2_{\beta,\gamma,p}+||\bar{Z}||^2_{\beta,\gamma,W}.
    $$
    Since the space is complete, the contraction theorem assures us the existence of a unique triplet $(Y,U,Z)$ in $\mathbb{L}^\beta$ such that $(Y,U,Z)=\Gamma(Y,U,Z)$, and $(Y,U,Z,K)$ is the solution to \eqref{eq:system_Y_both}, where $K$ is the one associated to $(Y,Z,U)$ by the map $\Gamma$. Since we know

\end{proof}

This last result generalizes the case of Brownian and Poisson noise, allowing for a more general structure in the jump part.

If we are interested only on a BSDE driven by a marked point process, the proof above still applies when the filtration $\mathbb{G}$ is generated only by $p$ and the data are adapted to it. Then we have the counterpart of theorem \ref{teo:both_processes}

\begin{teo}
    Let assumptions \ref{Ass:compensatore} and \ref{Ass:data}(i,ii,iv) hold for some $\beta>L_p^2+2L_f$, but with the data adapted to the filtration $\mathbb{G}$. Then the system \eqref{eq:system_Y_only_MPP} admits a unique solution in $L^{2,\beta}(A)\times L^{2,\beta}(p)\times\mathcal{I}^2$.
\end{teo}
\begin{proof}
    This is proven exactly as the case with also a Brownian motion. First, we show as in \ref{prop:given_generators_case}, the solution lies in $L^{2,\beta}(A)\times L^{2,\beta}(p)\times\mathcal{I}^2$ and, using Ito's formula, that it is unique. Next we build a contraction on this space, and obtain existence and uniqueness when the generator depends on $(Y,U)$.
\end{proof}
\begin{rem}
    A similar result does not hold in general in $L^1$. Counter examples are given in \cite{confortola2014backward}, where  additional hypotheses are then added to obtain an existence and uniqueness result. We also refer to \cite{confortola16LP} where the case $L^p$ is analysed.
\end{rem}

    \appendix
    \section{Some remarks on the Snell envelope theory}
    \label{app:snell_env}
    The Snell envelope theory has been treated in various works. \cite{el1981aspects}
    considers the case for a positive process without any restrictions on the filtration, obtaining general results.
    For a bit less general results, but still enough for our work, \cite{karatzas1998methods} develops
    the theory for non-negative càdlàg processes, while \cite{peskir2006optimal} treats the case where
    the process is càdlàg and left continuous over stopping times, and satisfies the condition
    \begin{equation}\label{eq:app_peskir_cond}\eval\left[\sup_t|\eta_t|\right]<\infty.\end{equation}
    The recent work \cite{kobylanski2012optimal} treats the subject in the framework of family of
    random variables indexed by stopping times, using quite general assumptions.
    In the following, let $(\Omega,\mathcal{F},\mathbb{P})$ be a probability space and let
    $\mathbb{F}=\left(\mathcal{F}_t\right)_{t\geq 0}$ be a filtration satisfying the usual conditions. Let $\eta$ be a cadlag process. Several properties that hold for positive processes can be shown under the condition \eqref{eq:app_peskir_cond}, as we will see in proposition \ref{res:ap_general_snell}.
     We  recall the following definition:
\begin{defi}
          An optional process $R$ of class [D] is said to be regular if $R_{t^-}=\leftidx{^p}R_t$ for any $ t<T $, where $\leftidx{^p}X$ indicates the predictable projection.
\end{defi}

    \begin{prop}
        \label{res:ap_general_snell}
        Let $\eta$ be a càdlàg process satisfying \eqref{eq:app_peskir_cond}. Define
        \begin{equation}
        \label{eq:app_snell_def}
        R_t=\essup\limits_{\tau \in \mathcal{T}_t}\econd{\eta_\tau}
        \end{equation}

        It holds that
        \begin{enumerate}[label={\roman*)},ref={\theprop.\roman*)}]
            \item \label{res:ap_snell_definition} $R_t$ is the Snell envelope of $\eta_t$. This means it is the smallest càdlàg supermartingale that dominates $\eta_t$, i.e. $R_t\geq \eta_t$ for all $t$ $ \mathbb{P} $-a.s.
            \item \label{res:ap_snell_optimcar} A stopping time $\tau^*$ is optimal in \eqref{eq:app_snell_def} (i.e. $R_t=\econd{\eta_{\tau^*}}$) if and only if one of the following conditions hold
            \begin{itemize}
                \item $
                R_{\tau^*}=\eta_{\tau^*} \text{ and } R_{s\wedge \tau^*} \text{is a }\mathbb{F}\text{-martingale}
                $
                \item $\eval[R_t]=\eval[\eta_\tau^*]$
            \end{itemize}
            \item \label{res:ap_snell_decomp} $R_t$ is of class [D], hence it admits decomposition
            $$
            R_t=M_t-K_t,
            $$
            where $M$ is a martingale, $K$ a predictable increasing process with $K_0=0$. $K$ can be decomposed as $K=K_t^c+K_t^d$, where $K^c$ indicates the continuous part and $K^d$ the discontinuous part. Moreover we have, a.s.

            \begin{gather*}
            \left\lbrace t:\Delta K_t>0\right\rbrace\subset\left\lbrace t: R_{t^-}=\eta_{t^-}\right\rbrace\\
        \text{or equivalently, }\quad   \Delta K_t=\Delta K_t\ind_{\{R(\eta)_{t^-}=\eta_{t^-} \}},
        \quad t\ge 0.
            \end{gather*}

            \item \label{res:ap_snell_regular} If the process $R_t$ is regular in the sense that $R_{t^-}=\leftidx{^p}R_t$, where $\leftidx{^p}R$ indicates the predictable projection, defining the stopping time
            $$
            D_t^*=\inf\lbrace s\geq t : R_s\neq M_s\rbrace,
            $$
            then $D_t^*$ is an optimal stopping time and it is in fact the largest optimal stopping time.
        \end{enumerate}
    \end{prop}
    \begin{proof}
        Define
        $$
        I=\inf\limits_{t\in[0,T]}\eta_t \qquad \text{ and } \qquad N_t=\econd{I},
        $$

        and since $\eta_t-I\geq 0$ for all $t$, we have $\eta_t-N_t\geq 0$ for all $t$.
         $N_t$ is a uniformly integrable martingale thanks to \eqref{eq:app_peskir_cond}. Consider $\tilde{\eta}_t=\eta_t-N_t\geq 0$ and $\tilde{R}_t=R_t-N_t$. Notice that then
        $$
        \tilde{R}_t=R_t-N_t=\essup\limits_{\tau\in\mathcal{T}_t}\econd{\eta_\tau-N_\tau}=\essup\limits_{\tau\in\mathcal{T}_t}\econd{\tilde{\eta}_\tau},
        $$
        i.e. $\tilde{R}$ is the Snell envelope of the positive process $\tilde{\eta}$. $R$ inherits all the properties from $\tilde{R}$. Let us see why the fourth property holds, as the rest are obtained similarly.
        If $R_t$ is regular, so is $\tilde{R}_t$ because we are adding a uniformly integrable martingale, which is regular (all uniformly quasi-left-continuous integrable càdlàg martingales are regular, see \cite{sheng1998semimartingale} Def 5.49). The result then holds by \cite{el1981aspects} pag 140.
    \end{proof}

%
%
%
%

\printbibliography
\end{document}